\documentclass[12pt]{article}
\usepackage[utf8]{inputenc}
\usepackage{amssymb, amsmath,amsthm, mathtools}
\usepackage{bbm}
\usepackage{amsfonts}
\usepackage[left=2cm,right=2cm,top=2cm,bottom=2cm]{geometry}
\usepackage{xcolor}

\usepackage{hyperref}

\usepackage{tabto}
\TabPositions{2 cm, 4 cm, 6 cm, 8 cm}

\usepackage{tikz-cd}

\usepackage{tikz}
\usetikzlibrary{arrows, automata, positioning}

\usepackage{tocloft}
\usepackage{appendix}

\usepackage{enumerate}

\usepackage{hyperref}

\hypersetup{
  colorlinks   = true, 
  urlcolor     = blue, 
  linkcolor    = blue, 
  citecolor   = red 
}
 
 \newtheorem{thmA}{Theorem}

\newtheorem{corA}[thmA]{Corollary}

\newtheorem{quA}[thmA]{Question}
 
\newtheorem{theorem}{Theorem}[section]
\newtheorem{corollary}[theorem]{Corollary}
\newtheorem{lemma}[theorem]{Lemma}

\newtheorem{proposition}[theorem]{Proposition}

\newtheorem*{theorem*}{Theorem}

\theoremstyle{definition}

\newtheorem{question}[theorem]{Question}

\newtheorem*{remark*}{Remark}
\newtheorem*{notation*}{Notation}
\newtheorem*{acks*}{Acknowledgements}

\renewcommand\leq{\leqslant}
\renewcommand\geq{\geqslant}

\newcommand{\Aut}{\operatorname{Aut}}
\newcommand{\Out}{\operatorname{Out}}
\newcommand{\Inn}{\operatorname{Inn}}

\newcommand{\MCG}{\operatorname{MCG}}

\newcommand{\cato}{\operatorname{CAT(0)}}

\begin{document}

\title{A note on separability in outer automorphism groups}
\author{Francesco Fournier-Facio}
\date{\today}
\maketitle

\begin{abstract}
We give a criterion for separability of subgroups of certain outer automorphism groups. This answers questions of Hagen and Sisto, by strengthening and generalising a result of theirs on mapping class groups.
\end{abstract}

\section{Introduction}

A subgroup $H \leq G$ is called \emph{separable} if for all $g \in G \setminus H$ there exists a finite quotient $q \colon G \to Q$ such that $q(g) \notin q(H)$. A question of Reid \cite[Question 3.5]{reid} asks whether all convex-cocompact subgroups of mapping class groups (as defined in \cite{farbmosher}, see also \cite{cc:char} for a characterisation) are separable. This was first verified for virtually cyclic subgroups \cite{mcg:sep}, and the full conjecture is known conditionally on the residual finiteness of hyperbolic groups \cite{BHMS}. Hagen and Sisto show in \cite{hagensisto} that certain examples of free convex-cocompact subgroups, constructed by Mj \cite{mj}, are separable. In order to do so, they prove the following criterion.

\begin{theorem*}[{\cite[Theorem 1.1]{hagensisto}}]
Let $g \geq 2$, and let $H \leq \MCG(\Sigma_g)$. Suppose that $H$ is torsion-free, malnormal, and convex-cocompact. If the preimage of $H$ under the natural quotient map $\MCG(\Sigma_g \setminus \{p\}) \to \MCG(\Sigma_g)$, for some point $p \in \Sigma_g$, is conjugacy separable, then $H$ is separable.
\end{theorem*}

\begin{remark*}
The theorem is stated for $g \geq 1$, however in that case $\MCG(\Sigma_g)$ is virtually free so every finitely generated subgroup is separable. We focus on the case $g \geq 2$ for coherence with the results of this note.
\end{remark*}

A group $G$ is said to be \emph{conjugacy separable} if for all non-conjugate elements $g, h \in G$ there exists a finite quotient $q \colon G \to Q$ such that $q(g)$ and $q(h)$ are non-conjugate. Our first result removes the other hypotheses on $H$, answering \cite[Question 1.4]{hagensisto}.

\begin{thmA}
\label{thm:mcg}

Let $g \geq 2$, and let $H \leq \MCG(\Sigma_g)$. If the preimage of $H$ under the natural quotient map $\MCG(\Sigma_g \setminus \{p\}) \to \MCG(\Sigma_g)$, for some point $p \in \Sigma_g$, is conjugacy separable, then $H$ is separable.
\end{thmA}

We can replace the hypothesis of conjugacy separability by a more geometric one, which goes in the direction of Reid's question \cite[Question 3.5]{reid}. This is the same way that Hagen and Sisto verify that their criterion holds for the groups constructed by Mj.

\begin{corA}
\label{cor:mcg}

Let $g \geq 2$, and let $H \leq \MCG(\Sigma_g)$ be a convex-cocompact subgroup. If the preimage of $H$ under the natural quotient map $\MCG(\Sigma_g \setminus \{p\}) \to \MCG(\Sigma_g)$, for some point $p \in \Sigma_g$, acts properly and cocompactly on a $\cato$ cube complex, then $H$ is separable.
\end{corA}

Notice that Corollary \ref{cor:mcg} applies to \emph{all} groups constructed in \cite{mj}, without needing to modify the construction to ensure malnormality, as in \cite[Section 5]{hagensisto}.

\medskip

Theorem \ref{thm:mcg} will follow from a more general result.

\begin{thmA}
\label{thm:general}
Let $G$ be a finitely generated group with trivial centre, and let $H \leq \Out(G)$. Suppose that $\Aut(G)$ is acylindrically hyperbolic and has no non-trivial finite normal subgroups. If the preimage of $H$ under the natural quotient map $\Aut(G) \to \Out(G)$ is conjugacy separable, then $H$ is separable.
\end{thmA}

Thanks to recent works proving acylindrical hyperbolicity of automorphism groups \cite{auto:oneend, auto:infend, auto:graph1, auto:graph2, auto:graph3}, this can be applied in several contexts. We isolate two instances.

\begin{corA}
\label{cor:hyperbolic}
Let $G$ be a torsion-free hyperbolic group, and let $H \leq \Out(G)$. If the preimage of $H$ under the natural quotient map $\Aut(G) \to \Out(G)$ is conjugacy separable, then $H$ is separable.
\end{corA}

\begin{corA}
\label{cor:raag}
Let $\Gamma$ be a finite simplicial graph that does not decompose as a join of two non-empty subgraphs. Let $A_\Gamma$ be the corresponding right-angled Artin group, and let $H \leq \Out(A_\Gamma)$. If the preimage of $H$ under the natural quotient map $\Aut(A_\Gamma) \to \Out(A_\Gamma)$ is conjugacy separable, then $H$ is separable.
\end{corA}

Both of these corollaries apply to $\Out(F_n)$. In this case, there are notions of convex-cocompact subgroups \cite{outf:cc1, outf:cc2, outf:cc3}; the corresponding preimage in $\Aut(F_n)$ is hyperbolic \cite{outf:cc1}, but we do not know of instances in which conjugacy separability of such groups is known (besides the case of cyclic subgroups).

\begin{quA}
Are there examples of subgroups $H \leq \Out(F_n)$ that are separable, not virtually cyclic, and satisfy some version of convex-cocompactness?
\end{quA}

We remark that free-by-cyclic groups $F_n \rtimes \mathbb{Z}$ often have non-separable subgroups \cite{monika}.

\medskip

The proof of Theorem \ref{thm:general} involves an element of $G$ that recognises non-inner automorphisms. The existence of such an element relies on the acylindrical hyperbolicity of $\Aut(G)$, and is what allows to strengthen and generalise the criterion of Hagen and Sisto, with a shorter proof. It partially answers \cite[Questions 1.5]{hagensisto} (Proposition \ref{prop:inner}) and \cite[Question 1.6]{hagensisto} (Proposition \ref{prop:random}).

\begin{acks*}
The author thanks Mark Hagen, Jonathan Fruchter, Monika Kudlinska, Alessandro Sisto, Ric Wade and Henry Wilton for useful comments on a first version; and the anonymous referee for more useful comments, especially for suggesting Proposition \ref{prop:random}.
\end{acks*}

\section{Proofs}

We will use $\pi$ to denote the quotient map $\pi \colon \Aut(G) \to \Out(G)$, and for a subgroup $H \leq \Out(G)$ we denote $\widetilde{H} \coloneqq \pi^{-1} H \leq \Aut(G)$. The starting point is the following criterion for separability, which in turn is based on Grossman's criterion for residual finiteness \cite{grossman}.

\begin{proposition}[{\cite[Proposition 2.5]{hagensisto}}]
\label{prop:criterion}
Let $G$ be a finitely generated group with trivial centre. Let $H \leq \Out(G)$ and $\alpha \in \Aut(G)$. Suppose that:
\begin{enumerate}
\item There exists $x \in G$ such that $\alpha(x) \neq h(x)$ for all $h \in \widetilde{H}$;
\item $\widetilde{H}$ is conjugacy separable.
\end{enumerate}
Then there exists a finite quotient $q \colon \Out(G) \to Q$ such that $q(\pi(\alpha)) \notin q(H)$.
\end{proposition}

Let us formulate the special case that we will use:

\begin{corollary}
\label{cor:criterion}
Let $G$ be a finitely generated group with trivial centre and let $H \leq \Out(G)$. Suppose that:
\begin{enumerate}
\item There exists $\gamma \in \Inn(G)$ such that the centraliser of $\gamma$ in $\Aut(G)$ is $\langle \gamma \rangle$;
\item $\widetilde{H}$ is conjugacy separable.
\end{enumerate}
Then $H$ is separable.
\end{corollary}

\begin{proof}
We show that Proposition \ref{prop:criterion} holds for $H$ and an arbitrary $\alpha \in \Aut(G) \setminus \widetilde{H}$. Choose $x \in G$ such that the corresponding inner automorphism $\gamma_x$ is as in the first assumption of the corollary. Suppose that $\alpha(x) = h(x)$ for some $h \in \widetilde{H}$. Then $h^{-1} \alpha (x) = x$, and so $h^{-1} \alpha$ belongs to the centraliser of $\gamma_x$ in $\Aut(G)$. By the choice of $x$, we have $h^{-1} \alpha \in \langle \gamma_x \rangle \leq \Inn(G) \leq \widetilde{H}$, and which contradicts $\alpha \notin \widetilde{H}$.
\end{proof}

Acylindrical hyperbolicity of $\Aut(G)$ ensures that the first item holds.

\begin{lemma}
\label{lem:witness}
Let $G$ be a group such that $\Inn(G)$ is infinite and $\Aut(G)$ is acylindrically hyperbolic and has no non-trivial finite normal subgroups. Then there exists $\gamma \in \Inn(G)$ such that the centraliser of $\gamma$ in $\Aut(G)$ is $\langle \gamma \rangle$.
\end{lemma}

In fact, such a $\gamma$ can be found by performing a simple random walk on $G$: see Proposition \ref{prop:random} and its proof.

\begin{proof}
Recall that a subgroup of an acylindrically hyperbolic group is called \emph{suitable} if it is non-elementary and does not normalise any non-trivial finite normal subgroup. In an acylindrically hyperbolic group with no non-trivial finite normal subgroups, every infinite normal subgroup is suitable \cite[Lemma 2.4]{osin}. In particular, $\Inn(G) \leq \Aut(G)$ is suitable. Therefore there exists an inner automorphism $\gamma \in \Inn(G)$ such that the elementary closure of $\gamma$ in $\Aut(G)$ is reduced to $\langle \gamma \rangle$ \cite[Lemma 5.6]{hull}. In particular, the centraliser of $\gamma$ in $\Aut(G)$ is reduced to $\langle \gamma \rangle$ \cite[Corollary 6.6]{DGO}.
\end{proof}

\begin{proof}[Proof of Theorem \ref{thm:general}]
Combine Lemma \ref{lem:witness} and Corollary \ref{cor:criterion}.
\end{proof}

For the next applications, we will use the following criterion to check that an automorphism group has no non-trivial finite normal subgroups. We say that $G$ has the \emph{unique root property} if $x^n = y^n$ for some $x, y \in G, n \geq 1$ implies $x = y$.

\begin{lemma}
\label{lem:root}
Let $G$ be a group with trivial centre and with the unique root property. Then $\Aut(G)$ has no non-trivial finite normal subgroups.
\end{lemma}

\begin{proof}
Suppose that $N \leq \Aut(G)$ is a finite normal subgroup. The action of $G \cong \Inn(G)$ on $N$ by conjugacy has a finite index kernel $K$. Then every element of $K$ commutes with every element of $N$, in other words, automorphisms in $N$ fix $K$ pointwise. Now let $\alpha \in N$ and $x \in G$. Let $n \geq 1$ be such that $x^n \in K$, so $x^n$ is fixed by $\alpha$. Then $x^n = \alpha(x^n) = \alpha(x)^n$, so by the unique root property $\alpha(x) = x$. This shows that $\alpha$ fixes every element of $G$ and we conclude.
\end{proof}

\begin{proof}[Proof of Corollary \ref{cor:hyperbolic}]
If $G$ is torsion-free elementary hyperbolic, then $G$ is either trivial or isomorphic to $\mathbb{Z}$, and in both cases $\Out(G)$ is finite, so all subgroups are separable.

If $G$ is torsion-free non-elementary hyperbolic, then $\Aut(G)$ is acylindrically hyperbolic \cite[Theorem 1.3]{auto:infend}. Moreover, $G$ has trivial centre and the unique root property \cite[Lemma 2.2]{root} and so $\Aut(G)$ has no non-trivial finite normal subgroups by Lemma \ref{lem:root}. Therefore Theorem \ref{thm:general} applies.
\end{proof}

\begin{proof}[Proof of Corollary \ref{cor:raag}]
Let $G = A_\Gamma$ be as in the statement. If $\Gamma$ has at most one vertex, then $G$ is either trivial or isomorphic to $\mathbb{Z}$, and in both cases $\Out(G)$ is finite, so all subgroups are separable.

If $\Gamma$ has at least two vertices, then $\Aut(G)$ is acylindrically hyperbolic \cite[Theorem 1.5]{auto:graph2}. Moreover, $G$ has trivial centre and the unique root property \cite[3-2) and 3-3)]{raag:magnus} and so $\Aut(G)$ has no non-trivial finite normal subgroups by Lemma \ref{lem:root}. Therefore Theorem \ref{thm:general} applies.
\end{proof}

For the results on mapping class groups, we apply Corollary \ref{cor:hyperbolic} to the special case of surface groups.

\begin{proof}[Proof of Theorem \ref{thm:mcg}]
Recall the Dehn--Nielsen--Baer Theorem: $\MCG(\Sigma_g)$ can be identified with an index-$2$ subgroup of $\Out(\pi_1(\Sigma_g))$, and $\MCG(\Sigma_g \setminus \{ p \})$ is the corresponding index-$2$ subgroup of $\Aut(\pi_1(\Sigma_g))$. Under these identifications, given a subgroup $H \leq \MCG(\Sigma_g)$, its preimage in $\MCG(\Sigma_g \setminus \{ p \})$ is the same as its preimage in $\Aut(\pi_1(\Sigma_g))$. Since $\pi_1(\Sigma_g)$ is torsion-free non-elementary hyperbolic, we can apply Corollary \ref{cor:hyperbolic} to get separability of $H$ in $\Out(\pi_1(\Sigma_g))$, which then implies separability in $\MCG(\Sigma_g)$.
\end{proof}

\begin{proof}[Proof of Corollary \ref{cor:mcg}]
By Theorem \ref{thm:mcg}, it suffices to show that, under the hypotheses, the preimage $\widetilde{H}$ is conjugacy separable. By convex-cocompactness of $H$, $\widetilde{H}$ is hyperbolic \cite{farbmosher, hamen}. By assumption $\widetilde{H}$ acts properly and cocompactly on a $\cato$ cube complex, and so it is virtually compact special \cite{agol, wise}. It follows that $\widetilde{H}$ is conjugacy separable \cite{special}.
\end{proof}

Let us end by addressing two questions from \cite{hagensisto}, which ask for elements that recognise non-inner automorphisms. The observation behind Lemma \ref{lem:witness} allows to answer both, under the assumptions of Theorem \ref{thm:general}. The two questions are asked for torsion-free acylindrically hyperbolic groups, with the case of hyperbolic groups being singled out. It is an open question whether the automorphism group of a finitely generated acylindrically hyperbolic group is always acylindrically hyperbolic \cite[Question 1.1]{auto:oneend}.

\begin{question}[{\cite[Question 1.5]{hagensisto}}]

Let $G$ be a torsion-free acylindrically hyperbolic group, and let $\phi_1, \ldots, \phi_n$ be non-inner automorphisms of $G$. Does there exist $x \in G$ with $x$ and $\phi_i(x)$ non-conjugate for all $i$?
\end{question}

\begin{proposition}
\label{prop:inner}

Let $G$ be a group such that $\Inn(G)$ is infinite and $\Aut(G)$ is acylindrically hyperbolic and has no non-trivial finite normal subgroups: for instance, a torsion-free non-elementary hyperbolic group. Then there exists $x \in G$ with the following property: for every non-inner automorphism $\phi$, $x$ and $\phi(x)$ are non-conjugate.
\end{proposition}

Recall that torsion-free non-elementary hyperbolic groups do indeed satisfy the hypotheses, as we saw in the proof of Corollary \ref{cor:hyperbolic}.

\begin{proof}
Let $x$ be such that $\gamma_x$ satisfies the statement of Lemma \ref{lem:witness}. Suppose that $\phi(x) = hxh^{-1}$. Then $\gamma_h^{-1} \phi$ fixes $x$, so it centralises $\gamma_x$. By the choice of $x$, we have $\gamma_h^{-1} \phi \in \langle \gamma_x \rangle \leq \Inn(G)$ and so $\phi \in \Inn(G)$.
\end{proof}

\begin{question}[{\cite[Question 1.6]{hagensisto}}]

Let $G$ be a torsion-free acylindrically hyperbolic group, let $\phi$ be a non-inner automorphism of $G$, and let $(w_n)$ be a simple random walk on $G$. Is it true that, with probability going to $1$ as $n$ goes to infinity, $w_n$ is not conjugate to $\phi(w_n)$?
\end{question}

Note that finite generation is implicit in this question, as \emph{simple} random walks are not defined over infinitely generated groups.

\begin{proposition}
\label{prop:random}

Let $G$ be a finitely generated group with trivial centre such that $\Aut(G)$ is acylindrically hyperbolic and has no non-trivial finite normal subgroups: for instance, a torsion-free non-elementary hyperbolic group. Let $(w_n)$ be a simple random walk on $G$. Then, with probability going to $1$ as $n$ goes to infinity, $w_n$ has the following property: for every non-inner automorphism $\phi$, $w_n$ and $\phi(w_n)$ are non-conjugate.
\end{proposition}

\begin{proof}
We will use a result from \cite{mahersisto}, from which we recall some terminology, in the special case we are interested in. We call an element $\alpha \in \Aut(G)$ \emph{asymmetric} if its elementary closure is $\langle \alpha \rangle$. Let $\mu$ be a probability distribution on $\Aut(G)$. We say that $\mu$ is \emph{admissible} (with respect to a fixed acylindrical action) if the support of $\mu$ is bounded and generates a non-elementary subgroup containing an asymmetric element. Let $\nu$ be the uniform measure on a finite generating set of $G$, and let $\mu$ be the pushforward of $\nu$ under the map $G \to \Aut(G)$. The simple random walk $(w_n)$ is generated by $\nu$, and it induces a random walk $(\gamma_{w_n})$ generated by $\mu$. Since the support of $\mu$ is finite, and it generates $\Inn(G)$ which is non-elementary and contains an asymmetric element (Lemma \ref{lem:witness}), $\mu$ is admissible. The cyclic subgroups $\langle \gamma_{w_n} \rangle$ are called \emph{random subgroups} of $\Aut(G)$ (for $k = 1$) in the language of \cite{mahersisto}.

Now we can apply \cite[Theorem 2.5]{mahersisto}, which states that with probability going to $1$ as $n$ goes to infinity, $\gamma_{w_n}$ is an asymmetric element of $\Aut(G)$. This implies that the centraliser of $\gamma_{w_n}$ is $\langle \gamma_{w_n} \rangle$ \cite[Corollary 6.6]{DGO}, and we conclude as in Proposition \ref{prop:inner}.
\end{proof}

\footnotesize

\bibliographystyle{amsalpha}
\bibliography{ref}

\vspace{0.5cm}

\normalsize

\noindent{\textsc{Department of Pure Mathematics and Mathematical Statistics, University of Cambridge, UK}}

\noindent{\textit{E-mail address:} \texttt{ff373@cam.ac.uk}}

\end{document}